\documentclass{article}
\usepackage{amssymb}

\usepackage{amsmath}
\usepackage[utf8]{inputenc}
\usepackage{epsfig}

\usepackage{subcaption}

\usepackage{graphicx}
\usepackage{mathrsfs}
\usepackage{bbm}
\usepackage{float}
\usepackage{pifont}
\usepackage{color}

\newtheorem{lemma}{Lemma}

\newtheorem{definition}{Definition}

\newtheorem{example}{Example}

\begin{document}

{\bf A dependent complex degrading system with non-periodic inspection times}

Inma T. Castro, Department of Mathematics, University of Extremadura, Spain \\
Luis Landesa, Escuela Politécnica, University of Extremadura, Spain

\begin{abstract}
This paper analyses a system subject to multiple dependent degradation processes. Degradation processes start at random times following a non homogeneous Poisson process and next dependently propagate. The growth of these degradation processes is modeled using gamma increments. We assume that the arrival of a new process to the system triggers the degradation rate of the processes present in the system. 
Under this framework, the analytic expression of the system reliability is obtained and bounds of the system reliability are also analyzed. 
Furthermore, the system is inspected at certain times. Information on the system health is recorded at these inspection times and the decision on performing maintenance actions on the system is taken at these times. We consider in this paper a dynamic inspection policy since the information that becomes available in an inspection time is taken into account to schedule the next inspection time. The maintenance cost for this system is dealt with the use of semi-regenerative process. Numerical examples are performed to illustrate the analytic expressions. 
\end{abstract}


\section{Introduction}
Industrial components can be subject to a corrosion process that involves several competing sources of degradation. A classical example is the pitting corrosion process. Pitting corrosion is defined as localized corrosion of a metal surface, confined to a point or small area that takes the form of cavities and it is considered one of the main causes of structural failure in industrial systems (\cite{Bhandari}). Pits usually initiate at random times and they grow according to the characteristics of the material. 
Pitting corrosion comprises two main processes: the initiation and the growth. Due to its stochastic nature, some probabilistic  models are developed to describe this corrosion process. Pit initiation can be described using a counting process such as a Poisson process. Pits growth is modelled using stochastic processes such as Markovian processes (\cite{Huynh2012}) or gamma processes (\cite{Zhu}). The stochastic modeling of Stress Corrosion Cracking (SCC) has been also analyzed under this double scheme of initiation and growth processes. SCC is a corrosion mechanism that forms cracks due to the combined influence of tensile stress and aggressive environment and it can be found in different components of nuclear power plants (\cite{Priya}).

In this paper, a system subject to several sources of degradation (defects) which appear at random times and next propagate is analyzed. This general double scheme appears in many practical cases such as Stress Corrosion Cracking (SCC) or pitting corrosion phenomenons. Depending on the characteristics of the system, different stochastic processes for the initiation and for the propagation processes can be chosen. In this paper, defects appear in the system following a non-homogeneous Poisson process and next propagate or degrade considering gamma increments. We call this probabilistic model as a NHPP-GP model. 
Different practical cases validate the NHPP-GP model. For example, \cite{Kuniewski} and \cite{Velazquez} used a NHPP-GP model to analyze the damage in oil industrial systems. On the other hand, \cite{Bordes} fitted a NHPP-GP model for data from the electricity generator EDF. Applications of the NHPP-GP scheme to model the Stress Corrosion Cracking can be found in \cite{Shafiee} for a three-bladed rotor system on an offshore wind turbine and in components of nuclear power plants (\cite{Blain} and \cite{Huynh2017} among others).

These previous papers assume that the defects are identical and they propagate independent of the rest of the defects. It could be the case for components in different production units without interference. However, \cite{Crowder} supported that the existence of common shared factors (same operational stresses, wear and tear history, materials quality among others) may indicate the possibility of dependence of the defects.  For example, interaction between adjacent defects has a significant influence on the propagation characteristics of the rest of the defects. Hence it is necessary to develop models capable of taking into account the interaction phenomena and where the interaction causes the defects growth changes. An interesting challenging problem is the study of the dependence for a NHPP-GP model and it is the main objective of this paper.  We call this model a ``dependent NHPP-GP" model. In the sequel of this work, by defects or degradation processes we mean the different competing sources of system degradation.

There are some works that deal with the analysis of the dependence between competing degradation processes. In the earlier work of \cite{Straub}, dynamic Bayesian networks are used to represent the dependencies between defects. In recent years, much attention has been placed on modelling the dependence between defects using copulas (\cite{Zhenyu}). Piecewise-deterministic Markov processes are also used to describe the dependence between degradation processes when physics-based models and multi-state models are used to describe the degradation evolution in certain structures (\cite{Lin}). Other approach to deal with the dependence between defects is to consider that some randomly occurring events can accelerate the degradation rate of the processes present in the system. These events can be, for example, shocks (\cite{Rafiee} and \cite{Song}) or even the system workload \cite{Che}. Following this approach, in this paper, a reliability model is developed for a system experiencing a dependent NHPP-GP scheme by considering the changing degradation rate due to the arrival of different defects to the system. It is motivated by the fact that, when a system is subject to multiple defects, the magnitude of the interaction between them is dependent on the number of defects present in the system (\cite{Kamaya}).

For degrading complex systems, different ways of defining the system failure can be used. For example, \cite{Huynh2017} assumes that the system can be considered as failed when the sum of the degradation levels of all the processes exceeds a predetermined value. In this paper, we assume that the system fails when the degradation level of a process exceeds a failure threshold (\cite{Caballe}). To avoid the system failure, an inspection strategy is implemented and
the decision of maintaining the system is taken on the basis of the observed condition of the system. 
Although monitoring can be performed with negligible times between inspections (\cite{Zhao}), in this paper we assume a discrete monitoring. There are many works in the literature in which the inspections are periodically performed regardless of the system state. The periodic inspection schemes, compared to non-periodic ones, are easier to implement (\cite{Golmakani}) but it can cause higher costs \cite{Barker}. Different models of non-periodic inspection strategies have been developed based on the residual useful life (\cite{Do}) and on the system degradation (\cite{Tai} and \cite{Grall}).  
In this paper, inspection times are scheduled taking into account the number of degradation processes in the system and their degradation levels.

In short,  the main contributions of this paper are the following:
\begin{itemize}
\item Extending the NHPP-GP model assuming dependence between defects. 
\item Analyzing the system reliability under a dependent NHPP-GP model. 
\item Proposing a maintenance policy with non-periodic inspection times. 
\item Building an analytic cost model for the policy using the semi-regenerative process theory. 
\end{itemize}

The paper is structured as follows. Sectiosn \ref{probabilisticmodelling} and \ref{timetofailure} describe the probabilistic model and analyze the system reliability. Some conditional probabilities are obtained in Section \ref{ConditionalProbability}. The non-periodic inspection policy is described in Section \ref{Inspectiontimes}. Sections \ref{ConditionalProbability} and \ref{optimalpolicy} model the system functioning and the maintenance strategy using a semi-regenerative process. Numerical examples are shown in Section \ref{numericalexamples} and Section \ref{conclusions} concludes.

\section{Probabilistic modeling} \label{probabilisticmodelling}

The assumptions of the model are the following. In the sequel of this work, we shall use the terms ``degradation processes'' and ``defects" interchangeably. 

\begin{enumerate}
\item We consider a system subject to different degradation processes or defects. Degradation processes start at random times following a Non Homogeneous Poisson process with intensity $\lambda(t)$. Let $\{N(t), t \geq 0\}$ be the process that governs this NHPP with cumulative function given by
\begin{equation} \label{acumuladalambda}
\Lambda(t)=\int_{0}^{t} \lambda(u)du. 
\end{equation}
Let $S_1, S_2, \ldots $ be the starting points of the NHPP. 
\item Each degradation process evolves according to a given rate that can undergo accelerations under the arrival of more degradation processes to the system. We assume that the growth of each degradation process depends on the number of process in the system.  So, if $S_{i} \leq s < t < S_{i+1}$,  $i \geq 1$, the density of the deterioration increment of the $k$-th degradation process in $(s,t)$ is given by
\begin{equation*}\label{incrementsame}
f_{\alpha c^{i-1}(t-s), \beta}(x)=\frac{\beta^{\alpha c^{i-1}(t-s)}}{\Gamma(\alpha c^{i-1}(t-s))}x^{\alpha c^{i-1}(t-s)-1}e^{-\beta x}, \quad x\geq 0, 
\end{equation*}
for $k=1, 2, \ldots, i$ and $x\geq 0$ and where $\Gamma$ denotes the gamma function defined as
\begin{equation} \label{gamma}
\Gamma(\alpha)=\int_{0}^{\infty} u^{\alpha-1}e^{-u} du. 
\end{equation}
This two stage process (NHPP for the initiation process and Equation (\ref{incrementsame}) for the growth process) is called a dependent NHPP-GP model with parameters $(\lambda(t), \alpha, \beta, c)$. 
\item The system fails when the degradation level of a process exceeds the failure threshold $L$. If the system fails, it stops working.
\item The system is inspected to check its status and, depending on its status, to perform a maintenance task. We assume that the time between inspections should always exceed a time $T_r$. Quantity $T_r$ represents the minimum time necessary to prepare the inspections. This time is used to arrange the tools and other materials necessary for the inspections
\item Inspections are non-sequentially performed in the following way. Let $T_1, T_2, \ldots, T_n$ be the inspection times for the system with $T_1=T$ and $T>0$. In an inspection time $T_i$, the number of degradation processes present in the system $N(T_i)$ and the degradation levels of each process are recorded.  Next inspection is planned taking into account this information. The system degradation at time $T_i$ can be expressed as
\begin{equation} \label{vectorX}
W\left(T_i\right)=\left(W_1(T_i), W_2(T_i), \ldots, W_{N(T_i)}(T_i)\right), 
\end{equation}
where $W_j(T_i)$ denotes the degradation level at time $T_i$ of the process that started at time $S_j$. With this recorded information, next inspection is scheduled at time 
  \begin{equation}
 T_{i+1}=T_i+m(T_i), 
 \end{equation}
where 
\begin{equation} \label{funcionm}
m(T_i)=\max \left( T_r, Tk^{N(T_i)} (1-\max(W(T_i))/M)\right), 
\end{equation}
with $k<1$, $M$ is a degradation level with $M<L$, $T_r$ corresponds to the minimum time between inspections with $T_r < T$ and $W\left(T_i\right)$ is given by Equation (\ref{vectorX}). 
 \item If the system is down in an inspection time, a corrective maintenance is performed and the system is replaced by a new one (corrective replacement). 
 \item In an inspection time $T_i$, if the maximum of the degradation levels of the processes exceeds $M$ but is less than $L$, a preventive replacement of the system is performed. A preventive replacement means the system replacement by a new one.  
 \item In an inspection time $T_i$, if the maximum of the degradation levels of the processes is less than $M$, the system remains in the state just as before the inspection time. 
 \item Inspections are assumed to be instantaneous, perfect and non-destructive.
\end{enumerate}

We recall that $S_1, S_2, \ldots, $ denote the starting points of the degradation processes. 

Let $X_i(t)$ be a gamma distribution with parameters $c^{i-1}\alpha t$ y $\beta$ for $i=1, 2, \ldots$ and $X_i^{(j)}(t)$ be the $j$-th replica of $X_i(t)$. Finally, we denote by $\left\{W_i(t), t \geq 0\right\}$ the degradation level of the process that started at time $S_i$ for $i=1, 2, \ldots$. 

The system degradation evolves as follows. 
\begin{itemize}
\item For $0 \leq t < S_1$, no degradation process is present in the system. 
\item For $S_1 \leq t < S_2$, the system is subject to one degradation process. Let $W_1(t)$ be the degradation level of this process at time $t$. Then
$$W_1(t)=X_1^{(1)}(t-S_1), \quad S_1 \leq t  \leq S_2,   $$
where $X_1^{(1)}(t-S_1)$ follows a gamma distribution with parameters $\alpha(t-S_1)$ and $\beta$. 
\item For $S_2 \leq t < S_3$, the system is subject to two degradation processes (with starting points $S_1$ and $S_2$). Then
\begin{align*}
W_1(t) &= X_1^{(1)}(S_2-S_1)+X_2^{(1)}(t-S_2), \\
W_2(t) &= X_2^{(2)}(t-S_2).  
\end{align*}
For fixed $(S_1, S_2)=(s_1, s_2)$, $X_1^{(1)}(s_2-s_1)$ follows a gamma distribution with parameter $\alpha(s_2-s_1)$ and $\beta$ and $X_2^{(1)}(t-s_2) (X_2^{(2)}(t-s_2))$ follows a gamma distribution with parameter $\alpha c(t-s_2)$ and $\beta$. 
\item For $S_3 \leq t < S_4$, the system is subject to three degradation processes (with starting points $S_1$, $S_2$ and $S_3$). Then
\begin{align*}
W_1(t) &= X_1^{(1)}(S_2-S_1)+X_2^{(1)}(S_3-S_2) + X_3^{(1)}(t-S_3),  \\
W_2(t) &= X_2^{(2)}(S_3-S_2)+X_3^{(2)}(t-S_3),  \\
W_3(t) &= X_3^{(3)}(t-S_3). 
\end{align*}
For fixed $(S_1, S_2, S_3)=(s_1, s_2, s_3)$ we get that $X_1^{(1)}(s_2-s_1)$ follows a gamma distribution with parameters $\alpha(s_2-s_1)$ and $\beta$, $X_2^{(1)}(s_3-s_2) (X_2^{(2)}(s_3-s_2))$ follows a gamma distribution with parameters $\alpha c(s_3-s_2)$ and $\beta$ and $X_3^{(i)}(t-s_3)$ follows a gamma distribution with parameters $\alpha c^2(t-s_3)$ and $\beta$ for $i=1, 2, 3$. Due to the additivity of the gamma distribution, $W_1(t)$ follows a gamma distribution with parameters 
$$\alpha_{1,3,t}=\alpha(s_2-s_1)+\alpha c (s_3-s_2)+\alpha c^2(t-s_3),$$
and $\beta$, $W_2(t)$ follows a gamma distribution with parameters 
$$\alpha_{2,3,t}=\alpha c (s_3-s_2)+\alpha c^2(t-s_3),$$
and $\beta$ and $W_3(t)$ follows a gamma distribution with parameters 
$$\alpha_{3,3,t}=\alpha c^2(t-s_3)$$ and $\beta$. 
\end{itemize}
In a general setting, the overall degradation of the process is given by
$$W(t)=(W_1(t), W_2(t), \ldots, W_{N(t)}(t))$$
where the processes $W_j(t)$ can be expressed as follows
\begin{align} \label{Wj}
W_j(t) &= \sum_{i=j}^{N(t)-1}X_i^{(j)}(S_{i+1}-S_i) + X_{N(t)}^{(j)}(t-S_{N(t)}), \quad 1 \leq j \leq N(t)-1 \\ \label{WN}
W_{N(t)}(t) &= X_{N(t)}^{(N(t))}(t-S_{N(t)}), 
\end{align}
and where $W_j(t)$ follows a gamma distribution with shape parameter
$$\alpha_{j,N(t),t}=\sum_{i=j}^{N(t)-1} \alpha c^{i-1} (S_{i+1}-S_i)+\alpha c^{N(t)-1}(t-S_{N(t)}),$$ 
and scale parameter $\beta$ with $j=1, 2, \ldots, N(t)$ and 
$W_{N(t)}(t)$ follows a gamma distribution with shape parameter
$\alpha_{N(t), N(t),t}=\alpha c^{N(t)-1}$ and scale parameter $\beta$. 
For fixed $N(t)=n$ ($n>1$) and considering a realization $(S_1, S_2, \ldots, S_n)=(s_1, s_2, \ldots, s_n)$ of the arrival process, then $W_j(t)$ given by Equations (\ref{Wj}) and (\ref{WN}) follows a gamma distribution with parameters 
$\alpha_{j,n,t}$ and $\beta$ where
\begin{equation}\label{alphajnt}
\alpha_{j,n,t}=\sum_{z=j}^{n-1} \alpha c^{z-1}(s_{z+1}-s_z)+c^{n-1}\alpha (t-s_n),  
\end{equation}
for $1 \leq j \leq n-1$ and
\begin{equation} \label{alphannt}
\alpha_{n,n,t}=c^{n-1}\alpha (t-s_n). 
\end{equation}
Notice that, if $c=1$, then $W_j(t)$ follows a gamma distribution with parameters 
\begin{align*} 
\alpha_{j,n,t} &=\sum_{z=j}^{n-1} \alpha(s_{z+1}-s_z)+\alpha (t-s_n) =  \alpha (t-s_j) \\ 
\alpha_{n,n,t} &= \alpha (t-s_n), 
\end{align*}
and $\beta$, that is, $\left\{W_j(t), \, t \geq 0\right\}$ is a gamma process with parameters $\alpha$ and $\beta$. 
\begin{example} \label{example1} The operating parameters given by \cite{Blain} are used. A two stage NHPP-GP process describes the Stress Corrosion Cracking of different components of nuclear power plants. Defects arrive at the system following a NHPP with intensity $\lambda=0.0002$ defects per unit time (hours). Growth process is modelled using a homogeneous gamma process with parameters $\alpha=0.0004$ (hours)$^{-1}$ and $\beta=1.5$ (mm)$^{-1}$. To show the characteristics of our model, a dependence parameter is given by $c=1.1$. Notice that a parameter dependence $c=1$ reduces our model to the model given by \cite{Blain}. Figure \ref{complexgamma} shows the length of the cracks versus time 

\begin{figure}
	\begin{center}
		\includegraphics[width=0.7\textwidth]{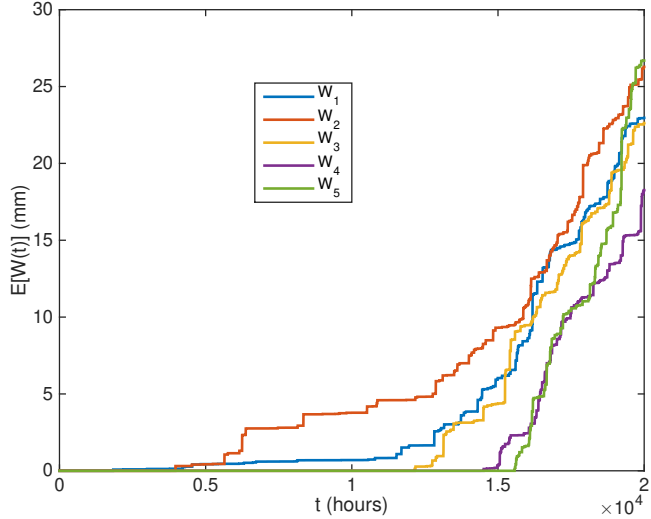}
		\caption{Simulation of a dependent NHPP-GP} \label{complexgamma}
	\end{center}
\end{figure}
\end{example}

In the sequel of this work, we denote by $f_{\alpha_{j,n,t}, \beta}$ ($F_{\alpha_{j,n,t}, \beta}$) the density (distribution) of a gamma distribution with parameters $\alpha_{j,n,t}$ and $\beta$
\begin{equation}\label{incrementsame}
f_{\alpha_{j,n,t}, \beta}(x)=\frac{\beta^{\alpha_{j,n,t}}}{\Gamma(\alpha_{j,n,t})}x^{\alpha_{j,n,t}-1}e^{-\beta x}, \quad x\geq 0, 
\end{equation}
where $\Gamma(\cdot)$ is given by Equation (\ref{gamma}). 

From \cite{Kuniewski}, the joint probability density of $(S_1, S_2, \ldots, S_n)$ given that $\left\{N(t)=n\right\}$ is equal to
\begin{equation}
f_{S_1,S_2,\ldots,S_n|N(t)}(s_1,s_2,\ldots,s_n|n)=\frac{n! \prod_{i=1}^{n} \lambda(s_i)}{\Lambda(t)^n}, 
\end{equation}
where $\Lambda(\cdot)$ is given by Equation (\ref{acumuladalambda}). Hence, the joint probability is
\begin{align*} \nonumber
f(s_1,s_2,\ldots,s_n, n)  & = P(S_1=s_1, S_2=s_2, \ldots, S_n=s_n, N(t)=n) \\ &= \exp(-\Lambda(t))\prod_{i=1}^{n} \lambda(s_i), \label{jointd}
\end{align*}
for $s_1 < s_2< \ldots < s_n<t$.

Expectations of the processes $\left\{W_j(t), t \geq 0\right\}$ for $j=1, 2, \ldots$ given by Equations (\ref{Wj}) and (\ref{WN}) are given next. 

\begin{lemma} \label{generalexpectations} The expectation of the process $\left\{W_j(t), t \geq 0\right\}$ given by Equations (\ref{Wj}) and (\ref{WN})  is equal to
\begin{equation} \label{expectationW}
\mathbb{E} \left[W_j(t)\right] = \sum_{n=j}^{\infty} \frac{\alpha c^{n-1}}{\beta} \displaystyle{\int_{0}^{t}} \frac{\Lambda(s)^n}{n!} \exp(-\Lambda(s)) ds, \quad t \geq 0, 
\end{equation}
for $j=1, 2, \ldots$. 
\end{lemma}
\begin{proof}
The proof is given in the Appendix. 
\end{proof}

{\it Remark.} Notice that (\ref{expectationW}) is equal to
\begin{equation*} 
\mathbb{E} \left[W_j(t)\right] = \sum_{n=j}^{\infty} \frac{\alpha c^{n-1}}{\beta} \int_{0}^{t} P(N(s)=n)ds.  
\end{equation*}
As we expected, expectation $\mathbb{E} \left[W_j(t)\right]$ is increasing with respect to $c$ for all $j$.

A closed-form expression for Equation (\ref{expectationW}) is obtained assuming that $\lambda(u)=\lambda$, that is, when the degradation processes arrive to the system following a homogeneous Poisson process. Firstly, we prove the following lemma. 

\begin{lemma} \label{exponentialexpectations} For $n \geq 1$ and $t>0$, 
\begin{equation} \label{intexp}
\int_{0}^{t} s^n \exp(-\lambda s) ds = \frac{n!}{\lambda^{n+1}} \left(1-\sum_{k=0}^{n} \exp(-\lambda t) \frac{(\lambda t)^k}{k!}\right). 
\end{equation}
\end{lemma}

\begin{proof}
The proof is given in the Appendix. 
\end{proof}

Assuming that $\lambda(t)=\lambda$ for all $t$ and using Lemma \ref{exponentialexpectations}, a closed-form expression for Equation (\ref{expectationW}) is obtained. The result is given in the following lemma. 

\begin{lemma} \label{meanlemma}
If the degradation processes start at random times following a non homogeneous Poisson process with parameter $\lambda$, then the expectation of the processes $\left\{W(t), t \geq 0\right\}$ shown in Equations (\ref{Wj}) and (\ref{WN}) are given by
\begin{align} \nonumber
\mathbb{E} \left[W_j(t)\right] &= \sum_{n=j}^{\infty} \frac{\alpha c^{n-1}}{\lambda \beta}P[N(t)>n], 
\end{align}
where $N(t)$ denotes the number of degradation processes at time $t$. 
\end{lemma}
\begin{proof}
The proof is given in the Appendix. 
\end{proof}

If the defects arrive to the system following a NHPP with parameter $\lambda$, then applying Lemma \ref{meanlemma}, the differences between expectations are given by
$$\mathbb{E} \left[W_{j}(t)\right]-\mathbb{E} \left[W_{j+1}(t)\right]=\frac{\alpha c^{j-1}}{\lambda \beta}P(N(t)>j). $$

\begin{example} Figure \ref{medias} shows the expected degradation (in mm) for different degradation processes versus $t$ for a dependent NHPP-GP with parameters $(\lambda, \alpha, \beta, c)$. The values of $\lambda$, $\alpha$ and $\beta$ are given by \cite{Blain}. Degradation processes start following a homogeneous Poisson process with parameter $\lambda=0.0002$ defects per hour.  Gamma increments are obtained using $\alpha=0.0004$ $(hours)^{-1}$ and $\beta=1.5$ $(mm)^{-1}$ with dependence parameter $c=1.1$. 
\end{example}
\begin{figure}
	\begin{center}
		\includegraphics[width=0.7\textwidth]{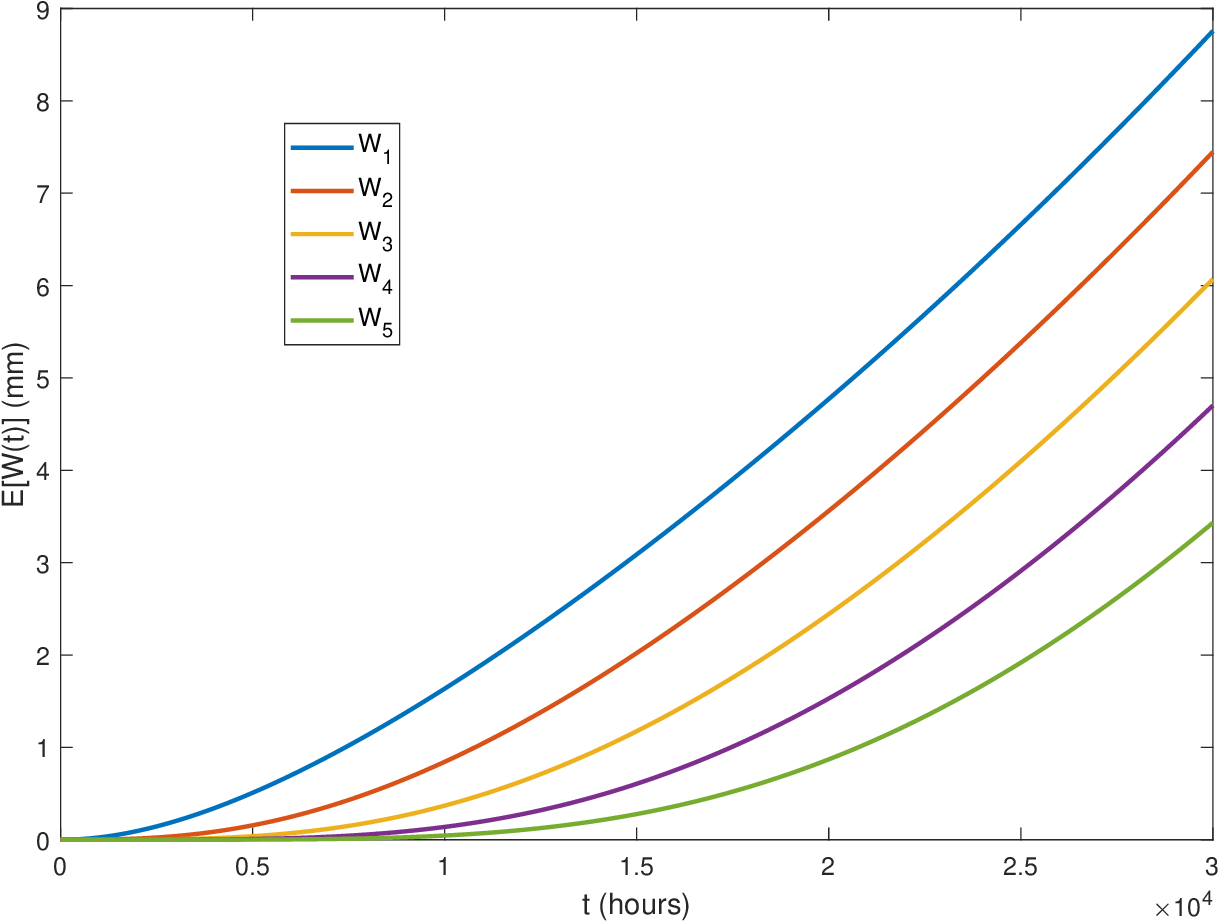}
		\caption{Average degradation for different processes versus $t$} \label{medias}
	\end{center}
\end{figure}
Next section analyzes the system reliability of a system subject to a dependent NHPP-GP model and gives some bounds of the system reliability. 
\section{System reliability} \label{timetofailure}
Starting at time $t=0$ with a new system $W(0)=0$, let $\sigma_{z|0}$ be the first hitting time to reach the level $z$ by the stochastic process $W(t)$. 
$$\sigma_{z|0}=\inf \left\{t \geq 0, \, max(W(t)) \geq z\right\}. $$ 
Lemma \ref{lemasurvival} gives the analytic expression of $\bar{F}_{\sigma_{z}|0}$, the survival function of $\sigma_{z|0}$. 
\begin{lemma} \label{lemasurvival}
The survival function of $\sigma_{z|0}(\cdot)$ is given by
\begin{equation} \label{Zsupervivencia}
\bar{F}_{\sigma_{z}|0}
(t)=P(\sigma_{z|0}\geq t)=\exp(-\Lambda(t)) \left(1+\bar{H}_{z}(t)\right). 
\end{equation}
where $\bar{H}_z(t)$ is given by
\begin{equation} \label{funcionG}
\bar{H}_z(t)=\sum_{n=1}^{\infty} \int_{0}^{t} \int_{s_1}^{t} \ldots \int_{s_{n-1}}^{t} \prod_{i=1}^{n} \lambda(s_i) F_{\alpha_{i,n,t},\beta}(z)ds_i, 
\end{equation}
for $n=1, 2, \ldots$ and $s_0=0$ where $F_{\alpha_{i,n,t},\beta}(\cdot)$ denote the distribution function of a gamma with parameters $\alpha_{i,n,t}$ and $\beta$ and where 
$\alpha_{i,n,t}$ is given by (\ref{alphajnt}). 
\end{lemma}
\begin{proof}  
The proof of this lemma is given in the Appendix.
\end{proof}
Equation (\ref{Zsupervivencia}) allows to compute the system reliability for a dependent NHPP-GP model. If the system fails when the maximum of the defects exceeds the failure threshold $L$, then the system reliability is given by
$$R(t)=\bar{F}_{\sigma_{L|0}}(t), \quad t \geq 0, $$
where $\bar{F}_{\sigma_{L|0}}(t)$ is obtained replacing $z$ by $L$ in Equation (\ref{Zsupervivencia}). 

Since the function (\ref{Zsupervivencia}) is tricky to use, some bounds of this function are given in Lemma \ref{lemasurvival} using a  
a result relating stochastic orders and gamma distributions given by \cite{Muler}. Firstly, the definition of stochastic order and likelihood ratio order is first recalled. 
\begin{definition} Let $X$ and $Y$ be two non negative random variables with probability density function $f_X$ and $f_Y$ with respect to the Lebesgue measure, cumulative distribution functions $F_X$ and $F_Y$ and survival functions $\bar{F}_X$ and $\bar{F}_Y$ respectively. 
\begin{itemize}
\item $X$ is said to be smaller than $Y$ in the usual stochastic order ($X \prec_{sto} Y$) if $\bar{F}_X \leq \bar{F}_Y$
(or $F_X \geq F_Y$ respectively). 
\item $X$ is said to be smaller than $Y$ in the likelihood ratio order ($X \prec_{lr} Y$) if $f_Y/f_X$ is non-decreasing on the union of the supports of $X$ and $Y$. 
\end{itemize}
\end{definition}
The likelihood ratio order implies the usual stochastic order. 

The result linking the likelihood ratio order and gamma distribution and given in \cite{Muler} (pag.62) is the following. 
 
\begin{lemma} \label{lemma1} Let $X$ e $Y$ be gamma distributed random variables with parameters $(a_1,b_1)$ and $(a_2,b_2)$, respectively, where $a_i,b_i>0$ for $i=1,2$. Then, if $a_1 \leq a_2$ and $b_1 \geq b_2$, then $X\prec_{lr} Y$. 
\end{lemma}
Since the likelihood ratio implies the usual stochastic order, we get that if $a_1 \leq a_2$ and $b_1 \geq b_2$, then $X\prec_{sto} Y$, that is
\begin{equation}
\bar{F}_X(t) \leq \bar{F}_Y(t) \quad \mbox{or equivalently} \quad {F}_X(t) \geq {F}_Y(t), \quad \forall t. 
\end{equation}
Lemma \ref{lemma1} is used to obtain bounds of the system reliability. 
\begin{lemma} \label{boundsurvival} Bounds of the survival function given by (\ref{Zsupervivencia}) are given by
\begin{equation}
\bar{F}_{\sigma_{z|0}}(t) \leq \exp\left(-\int_{0}^{t} \lambda(u)\bar{F}_{\alpha(t-u),\beta}(z)du\right), \quad t \geq 0
\end{equation}
for the upper bound and for the lower bound
\begin{equation}
\exp(-\Lambda(t))\left(1+\sum_{i=1}^{\infty} \frac{\left(\int_{0}^{t}\lambda(s)F_{Y_z}^{i-1}(t-s)ds\right)^i}{i!}\right) \leq \bar{F}_{\sigma_{z|0}}(t), \quad t \geq 0. 
\end{equation}
\end{lemma}
Notice that the lower bound corresponds to the same expression obtained by \cite{Caballe} in Lemma 2 for the survival function of a system subject to an independent NHPP-GP model ($c=1$).

\begin{proof} The proof is given in the Appendix. 
\end{proof}

\begin{example} Figure \ref{survivalf} shows the survival function $\bar{F}_{\sigma_z|0}$ for different values of $c$  and for $z=2$ mm. This figure has been obtained using parameters given by \cite{Blain} and used in previous examples ($\lambda=0.0002$ defects per hour, $\alpha=0.0004$ $(hours)^{-1}$ and $\beta=1.5 (mm)^{-1}$). The threshold $z=2$ is chosen since the cracks are detectable when their length reaches the detection threshold $z=2$ mm. Figure \ref{survivalf} corresponds to the probability that no defect has been detected in the system versus time. 
\end{example}
\begin{figure}
	\begin{center}
		\includegraphics[width=0.7\textwidth]{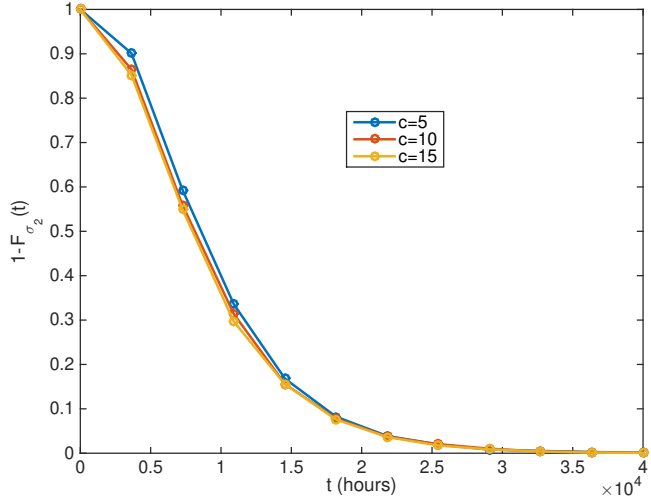}
		\caption{Survival function of crack detection} \label{survivalf}
	\end{center}
\end{figure}
Next section computes the system reliability for a system with an initial degradation. Lemmas \ref{lemasurvival} and \ref{survivalnonull} allow to compute the kernel of the stochastic process that describes the evolution of the system degradation. 
\subsection{System reliability in a degraded system} \label{reliabilitydegradedsystem}
Given $$W(0)=W_0=(x_1, x_2, \ldots, x_n),$$ let $\sigma_{z \mid x}$ be the first hitting time to the degradation level $z$, that is, 
$$\sigma_{z \mid x}=\inf \left\{t \geq 0, \, max(W(t)) \geq z\right\}. $$
\begin{lemma} \label{survivalnonull}
Given $W_0=x$, the survival distribution $\bar{F}_{\sigma_{z} \mid x}$ is given by
\begin{align*}
 & \bar{F}_{\sigma_{z}|x}(t) = \\
 &= exp(-\lambda(t))\left( \prod_{i=1}^{n}F_{\alpha c^{n-1}t, \beta}(z-x_i) \right. \\ 
 &+ \left. \sum_{p=1}^{\infty} \int_{0}^{t}\int_{s_{n+1}}^{t} \ldots \int_{s_{n+p-1}}^{t} \prod_{j=1}^{p} \lambda(s_{n+j})\prod_{i=1}^{n} F_{\alpha^*_{0,n+p,t},\beta}(z-x_i)\prod_{j=1}^{p}F_{\alpha^*_{n+j,n+p,t},\beta}(z) \right),  
\end{align*}
where $\alpha^*_{0,n+p,t}$ and $\alpha^*_{n+j,n+p,t}$ are given by
\begin{equation} \label{alphastar1}
\alpha_{0,n+p,t}^*=\alpha c^{n-1}s_{n+1}+\alpha c^{n}(s_{n+2}-s_{n+1})+\ldots+\alpha c^{n+p-1}(t-s_{n+p})
\end{equation}
and for $1 \leq j \leq p-1$
\begin{equation} \label{alphastar2}
\alpha_{n+j,n+p,t}^*=\sum_{i=j}^{p-1}\alpha c^{n+i-1}(s_{n+i+1}-s_{n+i})+\alpha c^{n+p-1}(t-s_{n+p}). 
\end{equation}
\end{lemma}
\begin{proof}
The proof is given in the Appendix. 
\end{proof}
\section{Non-periodic inspection times} \label{Inspectiontimes}
An inspection policy is implemented in this paper. As we detail in Section \ref{probabilisticmodelling}, non-periodic inspection times are considered. Following \cite{Barker} and \cite{Grall}, a scheduling function $m$ is defined and this function determines the amount of time until the next inspection. In this paper this scheduling function depends on the number of degradation processes in the system and also depends on the degradation levels of these processes. Let $T_i$ be the time of the $i$-th inspection. Then, given $T_i$ and recording the number of degradation processes at time $T_i$ and their degradation levels, next inspection is scheduled at time $T_{i+1}$ given by
$$T_{i+1}=T_i+m({W}(T_i)),  $$
where $m$ is given by (\ref{funcionm}) 
\begin{align*}
m({W}(T_i)) =\max \left(T_r, Tk^{N(T_i)} (1-\max(W(T_i))/M)\right), 
\end{align*}
whenever $\max(W(T_i))< M$. Quantity $T_r$ (with $0 <T_r \leq T$) represents the time required to prepare the maintenance facilities.  
Function $m(\cdot)$ is strictly increasing, with a minimum value equals to $T_r$ and maximum value equals to $T$. The minimum time between inspections prevents from the possibility of an infinite number of inspections on a finite time interval. 

Hence, in an inspection time, the next inspection time scheduling is given as follows.
 \begin{itemize}
\item If, in an inspection time, there is no degradation process in the system $N(T_i)=0$, we get that
$$\max \left(T_r, Tk^{N(T_i)} (1-\max(W(T_i))/M)\right)=T, $$
hence
next inspection is scheduled $T$ units of time after (consistent with the model). 
\item If, in an inspection time, there are $n$ degradation processes present in the system and the maximum of the degradation levels does not exceed $M$, the system is left as it is and next inspection time is scheduled at time 
$$\max \left(T_r, Tk^{n} (1-\max(W)/M)\right). $$
\item 
If, in an inspection time, the maximum of the degradation levels of the processes present in the system exceeds $M$ but it is less than the failure threshold $L$, then the system is preventively replaced by a new one and the inspection times are again scheduled based on this renewal.
\item If, in an inspection time, the system is not working, then the system is correctively replaced by a new one and the inspection times are again scheduled based on this renewal.
\end{itemize}

\begin{example}
Parameters given by \cite{Blain} are used. Defects start following a homogeneous Poisson process with parameter $\lambda=0.0002$ defects per hour. Growth process is modeled using the parameters $\alpha=0.0004 (\mbox{hour})^{-1}$, $\beta=1.5 (\mbox{mm})^{-1}$ and $c=5$. We assume that $M=5$ mm, first inspection is performed at time $T=3000$ hours and the minimum time to perform an inspection is equal to $T_r=1000$ hours. Figure \ref{meantimeinspectionsvaryk} shows the mean time between inspections (in hours) for different values of $k$. This plot has been performed using 30000 realizations in each point. 
\begin{figure}    
        \includegraphics[width=0.9\linewidth]{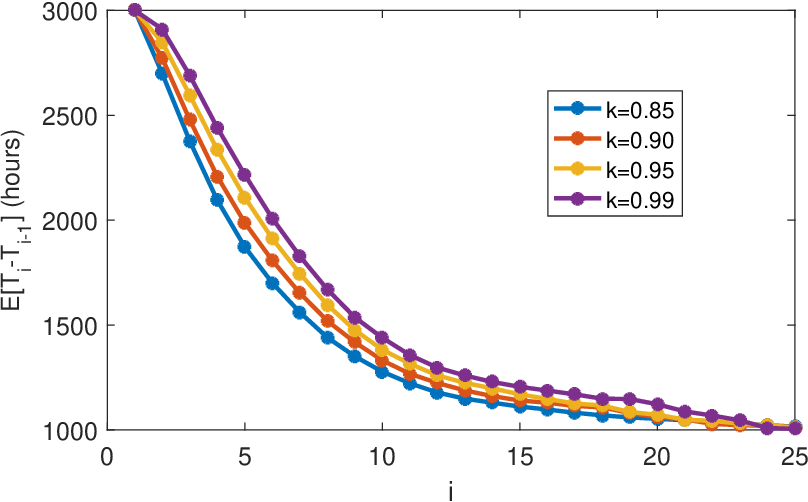}
    \caption{Mean time between inspections for different values of $k$}\label{meantimeinspectionsvaryk}
\end{figure}
\end{example}

\section{Stationary distribution} \label{ConditionalProbability}

Let $\left\{W(t), t \geq 0\right\}$ be the stochastic process that describes the system degradation at time $t$. Then, the process 
 $\left\{W(t), t \geq 0\right\}$ is a regenerative process with regeneration times the replacement times.
 
Furthermore, if the arrival of the degradation processes to the system follows a homogeneous Poisson process, the process  $\left\{W(t), t \geq 0\right\}$ is a semi-regenerative process with semi-regeneration times the inspection times. After each inspection, the system evolution depends on the system state in the inspection times. The process describing the system state after each inspection time 
$$Y_i={W}_{T_i}, \quad i=1, 2, \ldots, $$ is a Markov chain taking values in $[0,M) \times [0,M) \times \ldots [0,M)$ and $Y_0=0$. As the chain $\left\{Y_n, \, n=1, 2, \ldots \right\}$ comes back to 0 almost surely,  the existence of a stationary measure $\pi$ for $\left\{Y_n \right\}$ is proved (\cite{Grall}). This stationary measure is a solution of the invariance equation
 \begin{equation} \label{invarianceequation}
 \pi(\cdot)=\int_{0}^{M}\int_{0}^{M} \ldots \int_{0}^{M} \mathbb{Q}_{x}(\cdot) \pi(dx), 
 \end{equation}
 where $ \mathbb{Q}_x(\cdot)$ stands for the transition kernel of $\left\{Y_n \right\}$ with
 \begin{align} \label{kernel}
 \mathbb{Q}_{x}(dy) &=\mathbb{P}_x(W(T_{1}^{+}) \in dy) =P(W(T_{1}^{+}) \in dy \mid W(0^{+})=x),  
 \end{align}
where $T_1^{+}$ denotes the instant of time just after $T_1$. 
 
In the sequel of this section, we assume that the degradation processes start at random times following a homogeneous Poisson process with parameter 
$$\lambda(t)=\lambda, \quad t \geq 0. $$
Next, an expression for the transition kernel is obtained. To achieve this goal, given $N(0)=n$ with 
$${w}_0={W}(0)=(x_1,x_2,\ldots, x_n), \quad x_i \leq M$$ and $m_0=m(w_0)$ we first compute the following probability function 
\begin{align}   \label{probabilitydensity} 
& \mathbb{P}_{w_0}(dy_1,dy_2, \ldots, dy_{n+p}) = \\ &  \nonumber P({W}(m_0)^{-}\in (dy_1,dy_2, \ldots, dy_{n+p}) \mid {W(0)=w_0}).  
\end{align}

The following cases are envisioned taking into account the values of $n$ and $p$ and $(x_1, x_2, \ldots, x_n)$. 
\begin{itemize}
\item Case 1. $(x_1, x_2, \ldots, x_n)=0$, it means that the system is new, $N(0)=0$. With this initial condition, first inspection is performed at time $T_1=T$ and two scenarios are envisioned at the time of the next inspection $T_1$. 

\begin{enumerate}
\item No degradation process in $[0, T_1]$, that is $N(0)=N(T_1)=0$, $n=0$ and $p=0$.
\item $p$ degradation processes start to degrade in  $[0, T_1]$, that is $N(0)=0$, $N(T_1)=p$, $n=0$ and $p>0$.
\end{enumerate}
\item Case 2. $(x_1, x_2, \ldots, x_n)>0$ $n$ degradation processes at time 0 $N(0)=n$. With this initial condition, first inspection is performed at time $T_1=m(w_0)$ and two scenarios are envisioned at the time of the next inspection ($T_1$)
\begin{enumerate}
\item $N(T_1)=N(0)=n$, no degradation process starts in $[0,T_1]$, hence $p=0$.
\item $n+p$ degradation processes in  $T_1$, that is $N(0)=n$, $N(T_1))=n+p$, $n>0$ and $p>0$.
\end{enumerate}
\end{itemize}

With these four cases in mind, probability (\ref{probabilitydensity}) is computed. 
\begin{itemize}
\item Case 1.a) If $N(0)=N(T_1)=0$, probability (\ref{probabilitydensity}) is given by
\begin{equation} \label{probabilidad1}
P_0(dy)=\delta_0(dy) \exp(-\lambda T). 
\end{equation}
\item Case 1.b) If $N(0)=0$ and $N(T_1)=p>0$, we get that at time $T_1=T$, the degradation level for the $j$-th degradation process just before the inspection time is given by
\begin{align*}
W_j(T) &= \sum_{i=j}^{p-1} X_i^{(j)}(S_{i+1}-S_i)+X_p^{(j)}(T-S_p) \\
W_p(T) & = X_p^{(p)}(T-S_p), 
\end{align*}
where $1 \leq j \leq p-1$. Variable $W_j(T) $ follows a gamma distribution with shape parameter $\alpha_{j,p,T}$ where $\alpha_{j,p,T}$ is given by (\ref{alphajnt}) and scale parameter $\beta$. Hence, probability (\ref{probabilitydensity}) is given by
\begin{align} \label{probabilidad2}
& P_0(dy_1, dy_2, \ldots, dy_p) = \\ & \exp(-\lambda T) \left(\lambda^p \int_{0}^{T}\int_{s_{1}}^{T} \ldots \int_{s_{p-1}}^{T}  \prod_{j=1}^{p}  f_{\alpha_{j,p,T},\beta} (y_{j})  dy_{j} ds_{j} \right) . \nonumber
\end{align}
\item Case 2.a) If $N(0)=n$ and $N(T_1)=n$, it means that in $[0,T_1]$ $n$ degradation processes are developed. The degradation level for the $j$-th degradation process just before the inspection time $T_1=m_0$ is given by
\begin{align*}
W_j(m_0) &= x_j+X_n^{(j)}(m_0), \quad 1 \leq j \leq n, 
\end{align*}
where $X_n^{(j)}(m_0)$ follows a gamma distribution with parameters $\alpha c^{n-1} m_0$ and parameter $\beta$. Hence, probability (\ref{probabilitydensity}) is given by
\begin{align} \label{probabilidad3} 
 P_x(dy_1, dy_2, \ldots, dy_n) =  
 exp(-\lambda m_0)\prod_{i=1}^{n} f_{c^{n-1}\alpha m_0,\beta}(y_i-x_i)dy_i, 
\end{align}
\item Case 2.b) If $N(0)=n$ and $N(T_1)=n+p$, then in $[0,m_0]$ $n+p$ degradation processes are degraded. The degradation levels of these $n+p$ processes just before $T_1$ are equal to
\begin{align*}
W_j(m_0) &= x_j+\sum_{i=n}^{n+p-1} X_i^{(j)}(S_{i+1}-S_i)+X_{n+p}^{(j)}(m_0-S_{n+p}), 
\end{align*}
for $1 \leq j \leq n$ with $S_n=0$ and 
\begin{align*}
W_j(m_0) &= \sum_{i=j}^{n+p-1} X_i^{(j)}(S_{i+1}-S_i)+X_{n+p}^{(j)}(m_0-S_{n+p}) \\
\end{align*}
for $n+1 \leq j \leq n+p$ and
\begin{align*}
W_{n+p}(T_1) &= X_{n+p}^{(n+p)}(m_0-S_{n+p}).  \\
\end{align*}
\end{itemize}
And, 
\begin{align} \label{probabilidad4}
 & P_x(dy_1, dy_2, \ldots, dy_{n+p})   = \exp(-\lambda m_0)  \lambda^p \left(  \right. \\ & \left. \int_{0}^{m_0}\int_{s_1}^{m_0}\ldots \int_{s_{p-1}}^{m_0} g(y_i)dy_i \prod_{j=1}^{p}  f_{\alpha^*_{j,n+p,m_0},\beta}(y_{n+j})dy_{n+j} ds_{n+j} \right),  \nonumber
\end{align}
where 
$\alpha^*_{0,n+p,m_0}$ and $\alpha^*_{j,n+p,m_0}$ are given by (\ref{alphastar1}) and (\ref{alphastar2}) respectively and
$$g(y_i)=\prod_{i=1}^{n} f_{\alpha^*_{0,n+p,m_0},\beta}(y_i-x_i). $$

Next, we can compute the kernel $\mathbb{Q}_x(dy)$ given by (\ref{kernel}). \\

For $x=0$, we get that
\begin{align*}
\mathbb{Q}_0(dy) &= \delta_0(dy) \left(\exp(-\lambda T) +  \sum_{n=1}^{\infty} \int_{M}^{\infty} \int_{M}^{\infty} \ldots \int_{M}^{\infty} P_0(dy_1, dy_2, \ldots, dy_n) \right)
\end{align*}
where $P_0(dy_1, dy_2, \ldots, dy_n)$ is given by (\ref{probabilidad2})
\begin{align*}
\mathbb{Q}_0(dy_1, dy_2, \ldots dy_n) & = \exp(-\lambda T) \left(\lambda^{n}  \int_{0}^{T} \int_{0}^{T} \ldots \int_{s_{n-1}}^{T} \prod_{j=1}^{n}f_{\alpha_{j,n,T}, \beta}(y_j) \right) dy_j, 
\end{align*} 
for $y_j < M$ for $j=1, 2, \ldots, n$.  \\

For $x=(x_1,x_2, \ldots, x_n)$ we get that
\begin{align*}
\mathbb{Q}_x(dy) &= \delta_0(dy) \left(\int_{M}^{\infty} \int_{M}^{\infty} \ldots \int_{M}^{\infty} P_x(dy_1, dy_2, \ldots, dy_n) \right. \\  & \left. + \sum_{p=1}^{\infty}\int_{M}^{\infty} \int_{M}^{\infty} \ldots \int_{M}^{\infty} P_x(dy_1, dy_2, \ldots, dy_{n+p}) \right)
\end{align*}
with $P_x(dy_1, dy_2, \ldots, dy_n)$ and $P_x(dy_1, dy_2, \ldots, dy_{n+p})$ given by (\ref{probabilidad2}) and (\ref{probabilidad1}). Finally,  
\begin{align*}
& \mathbb{Q}_x(dy_1, dy_2, \ldots dy_n) =P_x(dy_1, dy_2, \ldots, dy_n) \\
& \mathbb{Q}_x(dy_1, dy_2, \ldots dy_n, dy_{n+1}, \ldots, dy_{n+p})=P_x(dy_1, dy_2, \ldots, dy_{n+p})
\end{align*}
where $P_x(dy_1, dy_2, \ldots, dy_n)$ and $P_x(dy_1, dy_2, \ldots, dy_{n+p})$ are given by (\ref{probabilidad3}) and (\ref{probabilidad4}).

Following the assumptions of this paper, the probability distribution of the semi-regenerative process $\pi$ fulfills Equation (\ref{invarianceequation}). The evaluation of $\pi$ is really tricky and requires to solve a multi-dimensional integral equation. To analyze the optimal maintenance strategy, the distribution of $\pi$ is simulated in Section \ref{numericalexamples}.

\section{Optimal maintenance policy} \label{optimalpolicy}
A maintenance policy is analyzed in this section. The long-run expected cost per unit time 
\begin{equation} \label{objectivecostfunction}
C_{\infty}=\lim_{t \rightarrow \infty}\frac{C(t)}{t}, 
\end{equation}
where $C(t)$ denotes the cumulative cost incurred in the time interval $[0,t]$ is chosen as objective cost function. The implementation of this objective cost function requires the evaluation of the stationary laws of the maintained system.

Let $R_1, R_2, \ldots, $ be the successive system replacements. The long-run expected cost per unit time is given by 
\begin{equation}
EC_{\infty}=\lim_{t \rightarrow \infty}  \frac{\mathbb{E}(C(t))}{t}=\frac{\mathbb{E}(C(R_1))}{\mathbb{E}(R_1)}, 
\end{equation}
that is, the asymptotic behaviour is focused on the first renewal. However, when the starting points of the degradation processes follow a homogeneous Poisson process, we can take advantage of the properties of the semi-regenerative process theory since $\left\{W(t), t \geq 0\right\}$ is a semi-regenerative process with semi-regeneration times the inspection times. 
The study of the asymptotic behaviour of $\left\{W(t), t \geq 0\right\}$ can be focused on a single semi-regenerative cycle (also known as Markov renewal cycle) defined by two successive inspection times and the long-run maintenance cost rate given by Equation (\ref{objectivecostfunction}) can be expressed as
\begin{equation}
C_{\infty}=\frac{\mathbb{E}_{\pi}[0,S_1]}{\mathbb{E}_{\pi}[0,S_1]}, 
\end{equation} 
where $\mathbb{E}_{\pi}$ denotes the $s$-expectation with respect to the stationary measure $\pi$ and $S_1$ denotes the length of the first Markov renewal cycle. Stationary measure fulfills Equation (\ref{invarianceequation}). The use of semi-regenerative processes to analyze the optimal maintenance policy is not new in the literature and they have been used by \cite{Berenguer} and \cite{Mercier} among others.

Integrating $\pi$, some measures related to the system can be calculated. So, 
denoting by $\mathbb{E}_{\pi}[N_p(T_1)]$ the number of preventive replacements in a semiregeneration cycle we get that
\begin{align} \label{preventivesemi}
 \mathbb{E}_{\pi}[N_p(T_1)] & =\pi(0) \left(F_{\sigma_M|0}(T)-F_{\sigma_L|0}(T) \right) \\ &+\int_{0}^{M}\int_{0}^{M} \ldots \int_{0}^{M} \pi({dx})\left(F_{\sigma_M|x}(m(x))-F_{\sigma_L|x}(m(x))\right), 
\end{align}
where $\bar{F}_{\sigma_M|x}(t)$ ($\bar{F}_{\sigma_L|x}(t)$) denotes the probability that, starting with $W_0=x$,  the maximum of the degradation levels does not exceed $M$ ($L$) at time $t$ 
$$\bar{F}_{\sigma_M|x}(t)=P_x(\max(W(t)) \leq M). $$ This function was calculated in Lemma \ref{survivalnonull} for $x\neq 0$ and in Lemma \ref{lemasurvival} for $x=0$. 

Denoting by $\mathbb{E}_{\pi}[N_c(T_1)]$ the number of corrective replacements in a semiregeneration cycle we get that
\begin{align}\label{correctivesemi}
 \mathbb{E}_{\pi}[N_c(T_1)] =\pi(0) F_{\sigma_L|0}(T) 
 +\int_{0}^{M}\int_{0}^{M} \ldots \int_{0}^{M} \pi({dx})F_{\sigma_L|x}(m(x)), 
\end{align}\label{downsemi}
with expected down time in a semiregeneration cycle given by
\begin{align}
 \mathbb{E}_{\pi}[d(T_1)] &=\pi(0) \int_{0}^{T} F_{\sigma_L|0}(u)du \\
 &+\int_{0}^{M}\int_{0}^{M} \ldots \int_{0}^{M} \pi({dx})\int_{0}^{m(x)}F_{\sigma_L|x}(u) du.  
\end{align}
Finally, the expected time of a semiregeneration cycle is given by
\begin{align}\label{expectedsemi}
 \mathbb{E}_{\pi}[T_1] = T \pi(0) + \int_{0}^{M}\int_{0}^{M} \ldots \int_{0}^{M} \pi(dx)m(x). 
\end{align}
Similar to the reasoning of \cite{Grall}, using $\pi$, the long-run expected cost per unit time can be expressed  as follows
\begin{align} \label{Cinfinito}
C_{\infty}(T,M) &= \frac{C_i \mathbb{E}_{\pi}[N_i(T_1)]}{\mathbb{E}_{\pi}[T_1]} + \frac{C_p \mathbb{E}_{\pi}[N_p(T_1)]}{\mathbb{E}_{\pi}[T_1]} + \frac{C_c \mathbb{E}_{\pi}[N_c(T_1)]}{\mathbb{E}_{\pi}[T_1]} + \frac{C_d \mathbb{E}_{\pi}[d(T_1)]}{\mathbb{E}_{\pi}[T_1]},  
\end{align}
where $\mathbb{E}_{\pi}[N_p(T_1)]$, $\mathbb{E}_{\pi}[N_c(T_1)]$, $\mathbb{E}_{\pi}[d(T_1)]$ and $\mathbb{E}_{\pi}[T_1]$ are given by (\ref{preventivesemi}:\ref{expectedsemi}). Finally, the number of inspections in a semi-regeneration cycle is given by
$$\mathbb{E}_{\pi}[N_i(T_1)] =1. $$
The optimization problem is formulated as
$$C(T_{opt}, M_{opt})=\inf\left\{C_{\infty}(T,M), \quad T \geq T_r, \quad 0 \leq M \leq L \right\}$$
where $C_{\infty}(T,M)$ is given by Equation (\ref{Cinfinito}), $M$ denotes the preventive threshold and $T$ denotes the time to the first inspection.

\section{Numerical examples} \label{numericalexamples}
In this section, some numerical examples are developed. To simplify the calculus, we assume that the system is subject to three degradation processes. These degradation processes arrive to the system according to a homogeneous Poisson process with parameter $\lambda=1$ defects per unit time. They grow with increments given by (\ref{incrementsame}) with parameters $\alpha=1$ time units$^{-1}$, $\beta=1$ ({degradation units})$^{-1}$
and $c=1.01$.  We assume that the system fails when the maximum of the degradation levels exceeds $L=8$ (degradation units). Inspections are scheduled according to a function $m(\cdot)$ given by (\ref{funcionm}) with $k=0.95$ and minimum time between inspections equals to $T_r=1$ time units. Each inspection incurs a cost of $C_i=50$ monetary units. We assume that a preventive replacement is performed when the system is working in an inspection time but the maximum of the degradation level of the processes exceeds $M$ with a cost of $C_p=300$ monetary units. If the system is failed in an inspection time, a corrective replacement is performed with an additional cost of $C_c=400$ monetary units. Between inspections, if the system is down, it incurs in a cost of $C_d=100$ monetary units per unit time. 

Using the stationary measure $\pi$, the long-run expected cost per unit time given by (\ref{Cinfinito}) is calculated. For that, the estimation of $\pi$ is computed. Notice that the distribution $\pi$ is composed of a discrete mass in $0$ and of a continuous part.  

To simulate $\pi(0)$, starting from $W_0=0$, 10000 realizations of the Markov chain $Y_n=W(T_n)$ for $n=1, 2, \ldots, 10000$ are computed and the estimation of $\pi(0)$ is computed as the frequency of times that the Markov chain visits the state 0 in these 10000 realizations. 

For example, we estimate the stationary distribution for the following example. We consider a dependent NHPP-GP model with the following parameters ($\lambda=1$, $\alpha=1$, $\beta=1$, $c=1.01$). The preventive threshold is given by $M=2$ degradation units. We assume that $T=3$ and $T_r=1$ (time units). Given $W_0=0$, 10000 realizations of the Markov chain are computed. From these realizations, 6939 times the Markov chain visits state $0$, hence
$$\widehat{\pi}_0=\frac{6939}{10000}=0.6939. $$
In these 10000 realizations, in 762 times, there is just one degradation process just after the inspection time. In 928 times, there are two degradation processes just after the inspection time and, in 1371 times, there are three degradation processes. 
Figure \ref{hist} shows the histogram for the degradation level just after the inspection time in these 762 times. Figure \ref{histogramados} shows the bivariate histogram for the degradation level just after the inspection time in these 928 times.

\begin{figure}
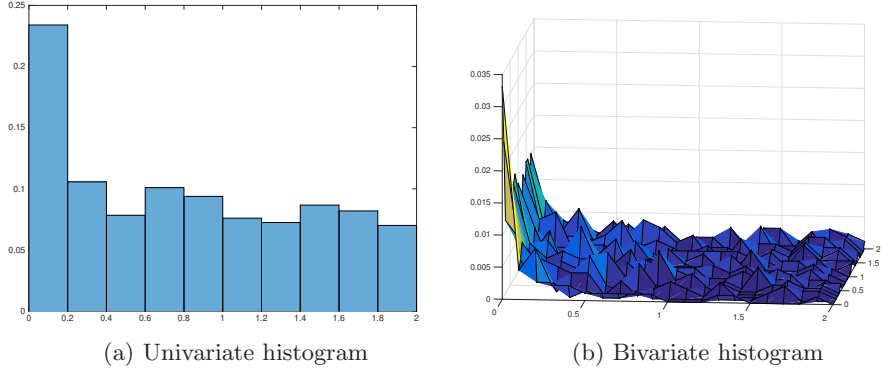

    \begin{subfigure}{.5\textwidth}
        \includegraphics[width=0.9\linewidth]{histogramauno.eps}
        \caption{Univariate histogram}\label{hist}
    \end{subfigure}
    \begin{subfigure}{.5\textwidth}
        \includegraphics[width=0.9\linewidth]{histogramados.eps}
        \caption{Bivariate histogram}\label{histogramados}
    \end{subfigure}
    \caption{Histograms for a sample of $\pi$}\label{histograms}
\end{figure}

%
%

Figure \ref{costeCM} shows the objective cost function given by Equation (\ref{Cinfinito}) versus $M$ and $T$ for this model. To compute the values of (\ref{Cinfinito}), 10000 replications of the Markov chain have been computed. With this distribution, stationary distribution $\pi$ is estimated. Using this estimation of $\pi$, the s-expectations $\mathbb{E}_{\pi}$ are computed using Monte-Carlo simulation.  For $T$, 6 values from 1 to 15 have been considered and for $M$, 5 values from 0.5 to 8 have been considered. For the calculus of  $\mathbb{E}_{\pi}$, 10000 values have been simulated for each combination of points and for each value of the estimation of $\pi$. 

\begin{figure}
	\begin{center}
		\includegraphics[width=0.7\textwidth]{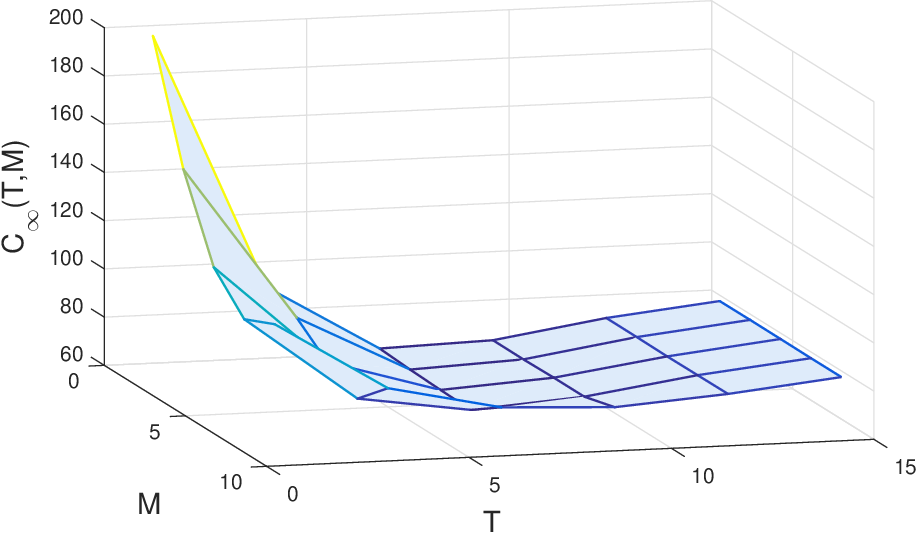}
		\caption{Objective cost function} \label{costeCM}
	\end{center}
\end{figure}

Considering this dataset, the optimal maintenance policy 
\begin{equation} \nonumber
C(T_{opt},M_{opt})=\min\left\{C(T,M), \, T_r \leq T, \quad 0 \leq M \leq L \right\},
\end{equation}
is obtained for $T_{opt}=6.6$ units of time and $M_{opt}=6.1250$ degradation units with a value of $C(T_{opt}, M_{opt})=62.2509$ monetary units per unit time. 

For the previous example, the parameter $k$ is fixed. However, we can analyze the optimal maintenance policy considering three parameters to optimize: $T$, $M$ and $k$. Hence, the objective cost function is equal to
\begin{equation} \nonumber
C(T_{opt},M_{opt}, k_{opt})=\min\left\{C(T,M,k), \, T_r \leq T, \quad 0 \leq M \leq L, \quad k \leq 1 \right\}. 
\end{equation}
As example of the three-dimensional optimization problem, a dependent NHPP-GP model with parameters $\lambda=0.75, \alpha=1, \beta=1, c=1.1$ is considered. The system fails when a degradation process exceeds $L=8$ degradation units. The minimum time between inspections is equal to $T_r=1$ time units. The sequence of costs is the following: the cost for a preventive replacement is equal to 300 monetary units, the cost for a corrective replacement is equal to 400 monetary units, the cost for an inspection is equal to 40 monetary units and the downtime cost is equal to $C_d=80$ monetary units per unit time. Figure \ref{costefixedk} shows the minimum expected cost rate versus $k$.  

 \begin{figure}
	\begin{center}
		\includegraphics[width=0.7\textwidth]{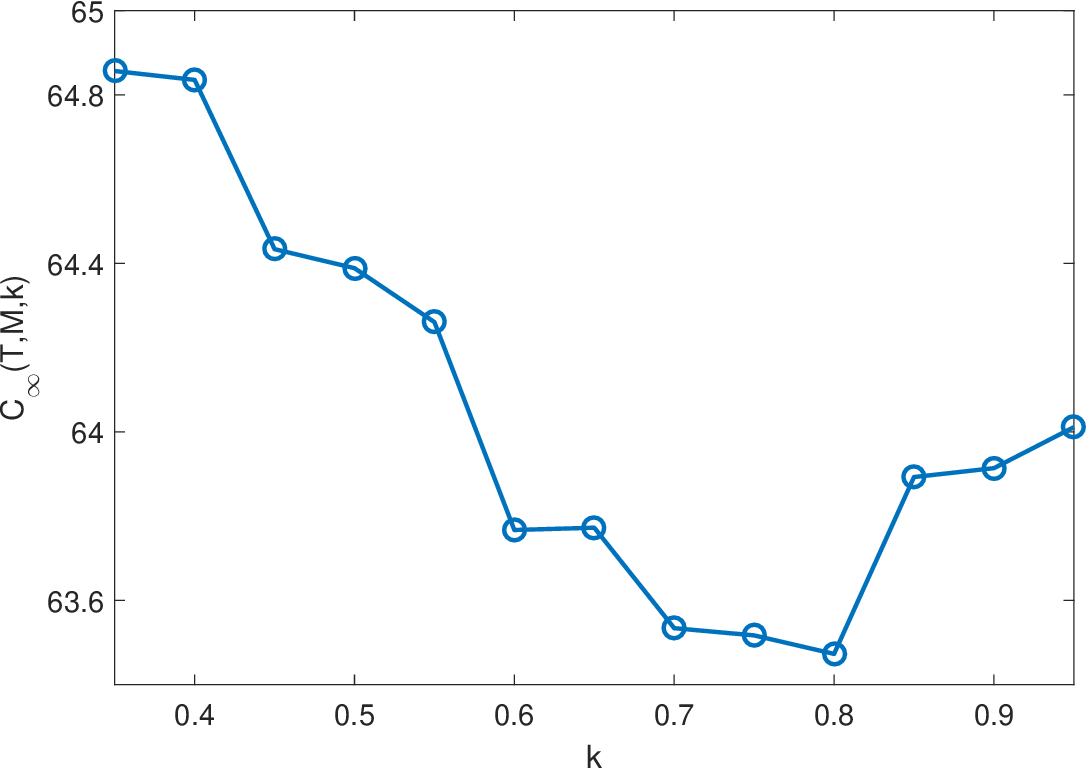}
		\caption{Minimum expected cost rate versus $k$} \label{variak}
	\end{center}
\end{figure}
As we can see, the optimal value is obtained for $k_{opt}=0.80$. For fixed $k=80$, Figure \ref{costefixedk} shows $C_{\infty}$ versus $T$ and $M$. This figure has been computed considering 6 values for $T$ from 1 to 10 and 5 values for $M$ from 0.5 to 8. The minimum values for the cost obtained for $T_{opt}$=6.4 and $M_{opt}=4.25$ with an expected cost equals to $C(T_{opt},M_{opt}, k_{opt})=63.4729$ monetary units per unit time. 
\begin{figure}
	\begin{center}
		\includegraphics[width=0.7\textwidth]{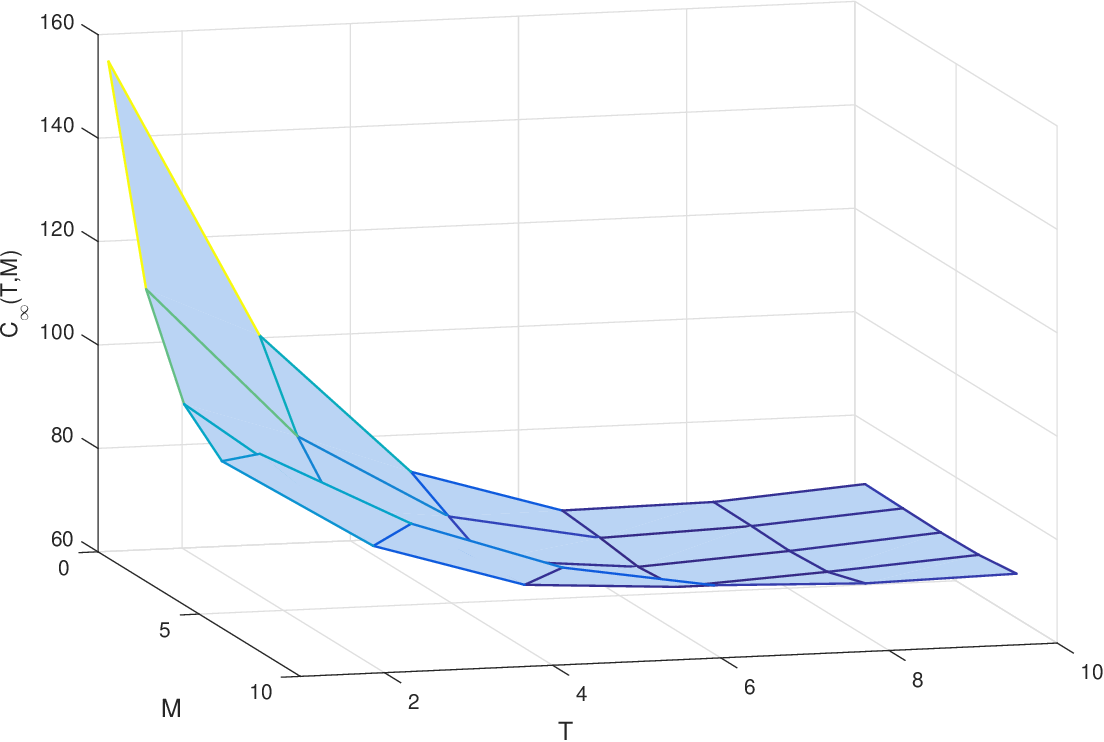}
		\caption{Expected cost rate for $k=0.80$} \label{costefixedk}
	\end{center}
\end{figure}

\section{Conclusions and further works}\label{conclusions}
This paper analyzes the system reliability for a system subject to many defects modelled using a NHPP-GP model. NHPP-GP model has been used in the literature to model, for example, the Stress Corrosion Cracking in an offshore wind turbine and for components of nuclear power plants. The novelty of this work is given by the assumption of dependence between the defects in a NHPP-GP model. We assume that the arrival of a new defect to the system triggers the degradation rate of the defects present in the system. With a system subject to different defects and modelled using a dependent NHPP-GP model, non periodic inspection times are scheduled. These inspections times are chosen according to the number of degradation processes present in the system and according to their degradation levels. An optimal maintenance policy is also analyzed considering techniques of semi-regenerative processes.

In this work, different defects start at random times and next propagate dependently according to a degradation law. It could be the case where the defects are not adjacent (\cite{Bordes}). However, when the defects are adjacent, they can coalesce to form a single critical defect with catastrophic consequences. It is well known that the coalescence of cracks plays a critical role in the growth process. Hence, an extension of this paper would be to consider a dependent NHPP-GP with coalescence between defects. 

In this work, we have considered the same degradation intrinsic process for all the degradation processes, that is, gamma increments. 
However, the property of equal growth process for all the defects may be violated in real world. For example, there are different degradation processes (cracking, deformation) that can develop on the pavement and whose degradation mechanism is different. An extension of this work would be to consider different degradation mechanisms between the processes and the implementation of a dependence structure between them. 

\section*{Acknowledgments}
This work is supported by Gobierno de Extremadura,
Spain (Project IB18073). For the first author,
this research was also supported by Ministerio de Econom{\'i}a y
Competitividad, Spain (Project PGC2018-094964-B-I00). 
 
\section*{Appendix}

\subsection*{Proof of Lemma \ref{generalexpectations}}
Taking expectations in each summation in (\ref{Wj}), we get that
\begin{align} 
& \mathbb{E}\left[\frac{\alpha c^{n-1} (S_{n+1}-S_n)}{\beta} \mathbf{1}_{\left\{S_{n+1}\leq t\right\}} \right] =    \frac{\alpha c^{n-1}}{\beta}\idotsint_{E_{n+1}} g(s_{n},s_{n+1}) \prod_{i=1}^{n+1} \lambda(s_i) ds_i  \label{int1}
\end{align}
where
\begin{eqnarray*}
E_{n+1}=\left\{(s_1,s_2,\ldots, s_n,s_{n+1}), \quad   0<s_1<s_2<\ldots<s_n<s_{n+1}<t\right\} 
\end{eqnarray*}
and
$$g(s_{n},s_{n+1})=(s_{n+1}-s_n)\exp\left\{-\Lambda(s_{n+1})\right\}. $$
Integrating (\ref{int1}), we obtain 
\begin{align}  \nonumber
 \mathbb{E}\left[(S_{n+1}-S_n) \mathbf{1}_{\left\{S_{n+1}\leq t\right\}} \right] &= \int_{0}^{t}\int_{0}^{s_{n+1}} g(s_n,s_{n+1})  \frac{\Lambda(s_n)^n}{n!} \prod_{i=n}^{n+1}\lambda(s_i)ds_i \\  \label{W11}
  &=  \int_{0}^{t} \frac{\Lambda(s)^{n+1}}{(n+1)!}\left(\exp(-\Lambda(s))-\exp(-\Lambda(t)\right) ds.  
  \end{align}
On the other hand, 
\begin{align} \nonumber
 \mathbb{E}\left[(t-S_n) \mathbf{1}_{\left\{S_{n}\leq t\leq S_{n+1}\right\}} \right] 
& = \idotsint_{E_{n}}(t-s_n)\exp\left\{-\Lambda(t)\right\}\prod_{i=1}^{n} \lambda(s_i) ds_i \\  \nonumber
& = \exp\left\{-\Lambda(t)\right\} \int_{0}^{t}(t-s)\frac{\Lambda(s)^n}{n!} ds \\    \label{W22}
& = \exp\left\{-\Lambda(t)\right\}   \int_{0}^{t}\frac{\Lambda(s)^{n+1}}{(n+1)!} ds  
\end{align} 
hence summing (\ref{W11}) and (\ref{W22}), (\ref{expectationW}) is fulfilled. 
\subsection*{Proof of Lemma \ref{exponentialexpectations}}
Induction is used to prove this lemma. For $n=1$, (\ref{intexp}) is fulfilled. 

We suppose that (\ref{intexp}) is true for $n$ and we have to prove that (\ref{intexp}) is true for $n+1$. So, we get that
\begin{align*}
 \int_{0}^{t}s^{n+1}\exp(-\lambda s)ds &= \frac{-t^{n+1}}{\lambda}\exp(-\lambda t)  + \frac{(n+1)}{\lambda}\int_{0}^{t}s^n \exp(-\lambda s)ds. 
\end{align*}
Using that (\ref{intexp}) is true for $n$, we get that
\begin{align*}
 \int_{0}^{t}s^{n+1}\exp(-\lambda s)ds &= \frac{-t^{n+1}}{\lambda}\exp(-\lambda t)  + \frac{(n+1)!}{\lambda^{n+1}}\left(1-\sum_{k=0}^{n} \exp(-\lambda t) \frac{(\lambda t)^k}{k!}\right) \\
& = \frac{(n+1)!}{\lambda^{n+2}}\left(1-\sum_{k=0}^{n+1} \exp(-\lambda t) \frac{(\lambda t)^k}{k!}\right), 
\end{align*}
and the result holds. 
\subsection*{Proof of Lemma \ref{lemasurvival}}
To compute the survival function, the number of defects present in the system at time $t$ is considered. 
\begin{align*}
 \bar{F}_{\sigma_{z}|0}(t) &= P(\sigma_z \geq t \mid W_0=0) \\
 &= P(max(W(t)) \leq  z \mid W_0=0) \\
& = \sum_{n=0}^{\infty} P\left(max(W(t)) \leq  z, \, N(t)=n \mid W_0=0\right). 
\end{align*}
It is trivial that  
\begin{eqnarray*} \label{Z0}
P(\sigma_z \geq t, \, N(t)=0 \mid W_0=0) = \exp(-\Lambda(t)). 
\end{eqnarray*}
For $N(t)=1$, we get that
\begin{align} \nonumber
 P(\sigma_z \geq t, \, N(t)=1 \mid W_0=0)  & = P(S_1 \leq t < S_2, \quad X_1^{(1)}(t-S_1) \leq z \mid W_0=0)  \nonumber \\
& =  \exp\left(-\Lambda(t)\right) \int_{0}^{t}\lambda(s) F_{\alpha(t-s),\beta}(z)ds,  \label{Z1}
\end{align}
where ${F}_{\alpha(t-s),\beta}$ denotes the distribution function of a gamma variable with parameters $\alpha(t-s)$ and $\beta$. 

In a general setting, 
\begin{align}  \nonumber
& P(\sigma_z \geq t \mid W_0=0) =\sum_{n=0}^{\infty} P(\sigma_z \geq t, N(t)=n \mid W_0=0) \\ \nonumber & = P(S_1 >t)+\sum_{n=1}^{\infty} P(S_n \leq t < S_{n+1}, \, max(W(t))<z) \\ \label{survival1}
&= \exp\left(-\Lambda(t)\right) \left(1+ \sum_{n=1}^{\infty} \int_{0}^{t}\int_{s_1}^{t} \ldots \int_{s_{n-1}}^{t} \prod_{i=1}^{n} \lambda(s_i) F_{\alpha_{i,n,t},\beta}(z) ds_i \right) 
\end{align}
with $s_0=0$ and where $F_{\alpha_{i,n,t}, \beta}$ denotes the distribution function of a gamma variable with parameters $\alpha_{i,n,t}$ given by (\ref{alphajnt})and $\beta$.

\subsubsection*{Proof of Lemma \ref{boundsurvival}}
Next, Lemma \ref{lemma1} is used to obtain bounds for the distribution function. 
For that, we consider the term $F_{\alpha_{i,n,t}}$ given in (\ref{Falpha}). Function  $F_{\alpha_{i,n,t}}$ corresponds to the distribution function of a gamma with parameters $\alpha_{i,n,t}$ and $\beta$ where
$$\alpha_{i,n,t}=\sum_{z=i}^{n-1}\alpha c^{z-1}(s_{z+1}-s_z)+c^{n-1}\alpha (t-s_n). $$
We get that the parameters $\alpha_{i,n,t}$ fulfill 
\begin{equation} \label{alfaparameters}
\alpha(t-s_i) \leq \alpha_{i,n,t} \leq   c^{n-1}\alpha (t-s_i), \quad t \geq s_i.  
\end{equation}
Applying Lemma \ref{lemma1} and inequality (\ref{alfaparameters}), due to 
likelihood ratio order implies the usual stochastic order, we get that
\begin{equation} \label{Falpha}
F_{\alpha c^{n-1}(t-s_i), \beta}(z) \leq F_{\alpha_{i,n,t},\beta}(z) \leq F_{\alpha(t-s_i),\beta}(z). 
\end{equation}
Replacing (\ref{Falpha}) in (\ref{funcionG}) , we get that
\begin{align} \nonumber
 \bar{F}_{\sigma_{z}|0}(t) & \leq exp(-\Lambda(t))\sum_{n=0}^{\infty} \frac{\left(\displaystyle{\int_{0}^{t}} \lambda(u)F_{\alpha(t-u),\beta}(z) du\right)^n}{n!} \\ \label{bound1}
& = exp\left(-\int_{0}^{t}\lambda(u)\bar{F}_{\alpha(t-u),\beta}(z) du\right), 
\end{align}
or, equivalently, 
$$ \bar{F}_{\sigma_{z}|0}(t) \leq  exp\left(-\int_{0}^{t}\lambda(u)F_{Y_z}(t-u) du\right), $$
where $F_{Y_z}$ denotes the distribution function of the first hitting time to the level $z$ for a homogeneous gamma process with parameters $\alpha$ and $\beta$
$$F_{Y_z}(t-u)=\frac{\displaystyle{\int_{z \beta}^{\infty}x^{\alpha (t-u)-1}e^{-\beta x} dx}}{\displaystyle{\int_{0}^{\infty}x^{\alpha (t-u)-1}e^{-\beta x} dx}}. $$
On the other hand, considering the lower bound of (\ref{Falpha}), we get that
\begin{align*}
 \bar{F}_{Y_{z}}(t)  & \geq  exp(-\Lambda(t))\left(1+\sum_{i=1}^{\infty} \frac{\left(\int_{0}^{t}\lambda(u)F_{c^{i-1} \alpha(t-u),\beta}(z)du\right)^i}{i!}\right),
\end{align*}
or equivalently, 
\begin{align*}
 \bar{F}_{\sigma_{z}|0}(t)   & \geq exp(-\Lambda(t))\left(1+\sum_{i=1}^{\infty} \frac{\left(\int_{0}^{t}\lambda(s) F_{Y_z}^{i-1} (t-s)ds\right)^i}{i!}\right), 
\end{align*}
where $\bar{F}_{Y_z}^{(i-1)}$ denotes the survival function for the first hitting time to the level $z$ for a homogeneous gamma process with parameters $c^{i-1}\alpha$ and $\beta$, that is, 
$$F_{Y_z}^{(i-1)}(t-s)=\frac{\displaystyle{\int_{z\beta}^{\infty}u^{c^{i-1}\alpha(t-s)-1} e^{-u} du}}{\displaystyle{\int_{0}^{\infty}u^{c^{i-1}\alpha(t-s)-1} e^{-u} du}}.$$

\subsection*{Proof of Lemma \ref{survivalnonull}}
Given $W_0=x$, we get that
\begin{align*}
 \bar{F}_{\sigma_{z}\mid x}(t) &= P(\sigma_z \geq t \mid W_0=x) \\
 &= P(max(W(t)) \leq  z \mid W_0=x) \\
& = \sum_{p=0}^{\infty} P\left(max(W(t)) \leq  z, \, N(t)=n+p \mid W_0=x\right). 
\end{align*}
If $N(t)=0$, $n$ degradation processes degrade in $[0,t]$ following a gamma process with parameters $c^{n-1} \alpha$ and $\beta$. Hence, 
$$P(\sigma_z \geq t, N(t)=n \mid W_0=x)=exp(-\lambda(t))\prod_{i=1}^{n}F_{\alpha c^{n-1}t, \beta}(z-x_i). $$
It $N(t)=n+p$, let $S_{n+1},S_{n+2}, \ldots, S_{n+p}$ be the starting points of the $p$ additional processes in  $[0,t]$. The degradation levels of the $n+p$ degradation processes at time $t$ is given by
\begin{align*}
W_i(t) &= x_i+X_n^{(i)}(S_{n+1})+X_{n+1}^{(i)}(S_{n+2}-S_{n+1})+\ldots+X_{n+p}^{(i)}(t-S_{n+p}),   
\end{align*}
for $1 \leq i \leq n$, 
\begin{align*}
W_{n+j}(t) &= \sum_{i=j}^{p-1}X_{n+i}^{(n+j)}(S_{n+i+1}-S_{n+i})+X_{n+p}^{(n+j)}(t-S_{n+p}), 
\end{align*}
for $n+1 \leq j \leq n+p-1$ and 
\begin{align*}
W_{n+p}(t) &= X_{n+p}^{(n+p)}(t-S_{n+p}). 
\end{align*}
So, for a realization $(S_{n+1}, S_{n+2}, S_{n+p})=(s_{n+1}, s_{n+2}, s_{n+p})$, we get that $W_i(t)$ follows a gamma distribution with parameters
\begin{equation} \label{alphastar1}
\alpha_{0,n+p,t}^*=\alpha c^{n-1}s_{n+1}+\alpha c^{n}(s_{n+2}-s_{n+1})+\ldots+\alpha c^{n+p-1}(t-s_{n+p})
\end{equation}
and $\beta$
for $1 \leq i \leq n$, $W_{n+j}(t)$ follows a gamma distribution with parameters
\begin{equation} \label{alphastar2}
\alpha_{n+j,n+p,t}^*=\sum_{i=j}^{p-1}\alpha c^{n+i-1}(s_{n+i+1}-s_{n+i})+\alpha c^{n+p-1}(t-s_{n+p})
\end{equation}
and $\beta$ for $1 \leq j \leq p-1$
and finally $W_{n+p}(t)$ follows a gamma distribution with parameters $\alpha c^{n+p-1}(t-s_{n+p})$ and $\beta$.

Hence, 
\begin{align*}
& P(\sigma_z \geq t, N(t)=n+p \mid W(0)=x)= exp(-\lambda(t)) \left( \right. \\ 
& \left.   \int_{0}^{t}\int_{s_{n+1}}^{t} \ldots \int_{s_{n+p-1}}^{t} \prod_{j=1}^{p} \lambda(s_{n+j})\prod_{i=1}^{n} F_{\alpha^*_{0,n+p,t},\beta}(z-x_i)\prod_{j=1}^{p}F_{\alpha^*_{n+j,n+p,t},\beta}(z) \right). 
\end{align*}

Finally, summing all the terms, we get that  
\begin{align*}
 & \bar{F}_{\sigma_{z} \mid x}(t) = \sum_{p=0}^{\infty}   P(\sigma_z \geq t, N(t)=n+p \mid W(0)=x) \\
 &= exp(-\lambda(t))\left( \prod_{i=1}^{n}F_{\alpha c^{n-1}t, \beta}(z-x_i) \right. \\ 
 &+ \left. \sum_{p=1}^{\infty} \int_{0}^{t}\int_{s_{n+1}}^{t} \ldots \int_{s_{n+p-1}}^{t} \prod_{j=1}^{p} \lambda(s_{n+j})\prod_{i=1}^{n} F_{\alpha^*_{0,n+p,t},\beta}(z-x_i)\prod_{j=1}^{p}F_{\alpha^*_{n+j,n+p,t},\beta}(z) \right). 
\end{align*}







\end{document}